\documentclass[13pt]{article}  
\usepackage{amssymb}              
\usepackage{amsthm}
\usepackage{eucal}
\usepackage{verbatim}
\usepackage[dvips]{graphicx}
\usepackage{multirow}
\usepackage{fancyhdr}
\usepackage{color}
\usepackage{enumerate}
\usepackage{calrsfs}
\usepackage{fullpage}
\usepackage{amsmath}
\newtheorem{theorem}{Theorem}
\newtheorem{lemma}{Lemma}
\newtheorem{definition}{Definition}
\newtheorem{remark}{Remark}

\makeatletter
\newcommand{\leqnomode}{\tagsleft@true}
\newcommand{\reqnomode}{\tagsleft@false}
\makeatother

\def\({\begin{eqnarray}}
\def\){\end{eqnarray}}
\def\[{\begin{eqnarray*}}
\def\]{\end{eqnarray*}}
\def\part#1#2{\frac{\partial #1}{\partial #2}}

\def\R{\mathbb{R}}
\def\N{\mathbb{N}}
\def\d{\mathrm{d}}
\def\tot#1#2{\frac{\d #1}{\d #2}}
\def\eps{\varepsilon}

\def\L{\mathcal{L}}

\def\P{P}     			
\def\M{\mathbb{M}}         	
\def\Mexp{\M_\mathrm{exp}} 	
\def\K{\mathbb{K}}         	

\def\cJH#1{\textcolor{blue}{\bf [#1]}}

\begin{document}

\title{Exponential asymptotic flocking in the Cucker-Smale model with distributed reaction delays}   
\author{Jan Haskovec \qquad Ioannis Markou}         
\date{}  
\maketitle

\begin{abstract}
We study a variant of the Cucker-Smale system with distributed reaction delays.
Using backward-forward and stability estimates
on the quadratic velocity fluctuations
we derive sufficient conditions
for asymptotic flocking of the solutions.
The conditions are formulated in terms
of moments of the delay distribution and
they guarantee exponential decay
of velocity fluctuations towards zero for large times.
We demonstrate the applicability
of our theory to particular delay distributions - exponential,
uniform and linear. For the exponential distribution,
the flocking condition can be resolved analytically,
leading to an explicit formula. For the other two distributions,
the satisfiability of the assumptions is investigated numerically.
\end{abstract}
\vspace{2mm}

\textbf{Keywords}: Cucker-Smale system, flocking, distributed time delay, velocity fluctuation.
\vspace{2mm}

\textbf{2010 MR Subject Classification}: 34K05, 82C22, 34D05, 92D50.
\vspace{2mm}

\section{Introduction}\label{sec:Intro}
Individual-based models of collective behavior attracted the interest of researchers
in several scientific disciplines. A particularly interesting aspect of the dynamics of multi-agent systems
is the emergence of global self-organizing patterns, while individual agents
typically interact only locally. This is observed in various types of systems -
physical (e.g.,  spontaneous magnetization and crystal growth in classical physics),
biological (e.g., flocking and swarming, \cite{Camazine, Vicsek-survey}) or socio-economical \cite{Krugman, Naldi-Pareschi-Toscani}.
The field of collective (swarm) intelligence also found many applications in engineering and robotics
\cite{Hamman, Jadbabaie}. The newest developments in the mathematical approaches to the field are captured in, e.g.,
\cite{ActiveParticles, Pareschi-Toscani-survey, ChHaLi, Ma,PiTr, PiVa18, PiVa, Ha1, Kalise, DHK19, LiuLi}.

The Cucker-Smale model is a
prototypical model of consensus seeking,
or, in physical context, velocity alignment.
Introduced in  \cite{CuSm1, CuSm2}, it has been extensively
studied in many variants, where the main point of interest
is the asymptotic convergence of the (generalized) velocities
towards a consensus value.
In this paper we focus on a variant of the Cucker-Smale model with distributed delay.
We consider $N \in\mathbb{N}$ autonomous agents described by their phase-space
coordinates $(x_i(t), v_i(t))\in\mathbb{R}^{2d}$, $i=1,2,\cdots,N$,
$t\geq 0$, where $x_i(t) \in \mathbb{R}^{d}$, resp. $v_i(t) \in
\mathbb{R}^{d}$, are time-dependent position, resp. velocity,
vectors of the $i$-th agent,
and $d\geq 1$ is the physical space dimension.
The agents are subject to the following dynamics
\(
      \dot{x}_{i} &=& v_{i} \label{CS_delay1} \\
      \dot{v}_{i} &=& \frac{\lambda}{N}\sum\limits_{j=1}^N
       \int_0^\infty   \psi(|{x}_{i}(t-s)-{x}_{j}(t-s)|)
      ({v}_{j}(t-s)-{v}_{i}(t-s)) \d\P(s), \label{CS_delay2}
\)
for $i=1,2,\cdots,N$, where $|\cdot|$ denotes the Euclidean distance in $\mathbb{R}^d$.
The parameter $\lambda>0$ is fixed
and $\P$ is a probability measure on $[0,\infty)$.
For simplicity we consider constant initial datum on $(-\infty,0]$
for the position and velocity trajectories,
\( \label{CS_delay_IC}
     (x_{i}(t),v_{i}(t)) \equiv (x_{i}^0,v_{i}^0) \qquad \mbox{for } t \in (-\infty,0],
\)
with $(x_{i}^0, v_{i}^0) \in \mathbb{R}^{d}\times \mathbb{R}^{d}$
for $i=1,2,\cdots,N$.
The function $\psi: [0,\infty) \to (0,\infty)$ is a
positive nonincreasing differentiable function that models the communication rate
between two agents $i$, $j$, in dependence of their 
metric distance. For notational convenience, we shall denote
\[
   \psi_{ij}(t):=\psi(|x_{i}(t)-x_{j}(t)|).
\]
In our paper we shall introduce the following three assumptions on $\psi=\psi(r)$,
namely, that
\( \label{ass:psi0}
   \psi(r) \leq 1 \qquad\mbox{for all } r\geq 0,
\)
which clearly does not restrict the generality due to the freedom to choose
the value of the parameter $\lambda>0$.
Moreover, we assume that there exist some $\gamma<1$ and $c, R>0$ such that
\( \label{ass:psi1}
   \psi(r) \geq c r^{-1+\gamma} \qquad\mbox{for all } r\geq R,
\)
and that there exists $\alpha>0$ such that
\( \label{ass:psi2}
   \psi'(r) \geq -\alpha \psi(r) \qquad\mbox{for all } r > 0.
\)
The prototype rate considered by Cucker and Smale in \cite{CuSm1, CuSm2}
and many subsequent papers is of the form
\begin{equation}
\label{CS_cw}
   \psi(r)=\frac{1}{(1+r^2)^{\beta}},
\end{equation}
with the exponent $\beta\geq 0$.
The assumption \eqref{ass:psi1} is verified for \eqref{CS_cw} if $\beta<1/2$,
while assumption \eqref{ass:psi2} is satisfied for all $\beta\geq 0$ by choosing $\alpha:=2\beta$.
Let us point out that the results of our paper are not restricted to the particular form \eqref{CS_cw}
of the communication rate.

In real systems of interacting agents - animals, humans or robots,
the agents typically react to the information perceived from their surroundings with positive
processing or reaction delay, which might have a significant effect
on their collective behavior.
The system \eqref{CS_delay1}--\eqref{CS_delay2} represents a model
for flocking or consensus dynamics where the reaction (or information processing)
delay is distributed in time according to the probability distribution $\P$.
The main objective in the study of Cucker-Smale type models is their
asymptotic behavior, in particular, the concept of \emph{conditional or unconditional
flocking}. In agreement with \cite{CuSm1, CuSm2}
and many subsequent papers,
we say that the system exhibits flocking behavior if there is
asymptotic alignment of velocities and the particle group stays
uniformly bounded in time.

\begin{definition} \label{def:flocking}
We say that the particle system \eqref{CS_delay1}--\eqref{CS_delay2}
exhibits \emph{flocking} if its solution $(x(t),v(t))$ satisfies
\begin{equation*}
   \lim_{t \to \infty} |v_i-v_j| =0, \qquad
   \sup_{t \geq 0} |x_i-x_j| < \infty,
\end{equation*}
for all $i,j = 1,2,\cdots,N$.
\end{definition}

The term \emph{unconditional flocking} refers to
the case when flocking behavior takes place for all initial
conditions, independently of the value of the parameters $\lambda>0$
and $N\in\mathbb{N}$. The celebrated result of Cucker and Smale
\cite{CuSm1, CuSm2} states that the system
\eqref{CS_delay1}--\eqref{CS_delay2} without delay
(this corresponds to the formal choice $\d \P(s):=\delta(s) \d s$,
with $\delta$ the Dirac delta measure)
with the communication rate \eqref{CS_cw} exhibits unconditional
flocking if and only if $\beta < 1/2$. For $\beta \geq 1/2$ the
asymptotic behavior depends on the initial configuration and the
particular value of the parameters $\lambda>0$ and $N\in\mathbb{N}$.
In this case we speak about \emph{conditional flocking}. The proof
of Cucker and Smale (and its subsequent variants, see \cite{HaTa, HaLiu,
CaFoRoTo}) is based on a bootstrapping argument, estimating, in
turn, the quadratic fluctuations of positions and velocities,
and showing that the velocity fluctuations decay
monotonically to zero as $t\to\infty$.

The presence of delays in \eqref{CS_delay1}--\eqref{CS_delay2}
introduces a major analytical difficulty.
In contrast to the classical Cucker-Smale system (without delay),
the quadratic velocity fluctuations are, in general, not decaying
in time, and oscillations may appear.
In fact, oscillations are a very typical phenomenon exhibited by solutions
of differential equations or systems with delay, see, e.g., \cite{Gyori-Ladas}.
In \cite{HasMar} we developed an analytical approach
for the Cucker-Smale model with lumped delay
(corresponds to the formal choice $\d \P(s):=\delta(s-\tau) \d s$ in \eqref{CS_delay1}--\eqref{CS_delay2},
with a fixed $\tau>0$).
It is based on the following two-step procedure: first, construction of a Lyapunov functional,
which provides global boundedness of the quadratic velocity fluctuations.
Second, forward-backward estimates on appropriate quantities
that give exponential decay of the velocity fluctuations.
The main goal of this paper is to generalize the approach
to the case of distributed delays with a general probability measure $P$.
A demonstration of the approach to the scalar negative feedback
equation with distributed delay, which can be seen as a special case
of \eqref{CS_delay1}--\eqref{CS_delay2} with $\psi\equiv 1$ and $N=2$, was recently
given in \cite{Has}.


Flocking in Cucker-Smale type models with fixed lumped delay and renormalized
communication weights was recently studied in \cite{LiWu, ChHa1}.
Both these papers consider the case where the delay in the
velocity equation for the $i$-th agent is present only in the
$v_j$-terms for $j\neq i$.
This allows for using convexity arguments
to conclude a-priori uniform boundedness of the velocities.
Such convexity arguments are not available for our
system \eqref{CS_delay1}--\eqref{CS_delay2}.
In \cite{ChHa2} the method is extended to the mean-field limit ($N\to\infty$) of the model.
In \cite{ChLi} the authors consider heterogeneous delays both in the
$x_j$ and $v_j$ terms and they prove asymptotic flocking for small
delays and the communication rate \eqref{CS_cw}. A system with time-varying
delays was studied in \cite{PiTr}, under the a-priori assumption
that the Fiedler number (smallest positive eigenvalue) of the
communication matrix $(\psi_{ij})_{i,j=1}^N$ is uniformly bounded
away from zero. The same assumption is made in \cite{ErHaSu}
for a Cucker-Smale type system with delay and multiplicative noise.
The validity of this relatively strong assumption
would typically be guaranteed by making the communication rates
$\psi_{ij}$ a-priori bounded away from zero, which excludes the
generic choice \eqref{CS_cw} for $\psi$. Our approach does not require
such a-priori boundedness.

Cucker-Smale systems with distributed delays were
studied in \cite{LiuLi} and \cite{PiVa}.
In both works, the delay is present in the expression for $v_j$ only,
while $v_i$ in \eqref{CS_delay2} is evaluated at the present time $v_i(t)$.
The $L^\infty$ analysis in \cite{PiVa} is based on a system of dissipative
differential inequalities for the position and velocity diameters,
leading to a nonexplicit ``threshold on the time delay''.
The work \cite{PiVa} introduces hierarchical leadership
to the distributed delay system. For the case of free will
ultimate leader (i.e., it can change its velocity freely),
a flocking result is given under a smallness condition on the leader's acceleration.
To our best knowledge, the Cucker-Smale system of the form \eqref{CS_delay1}--\eqref{CS_delay2},
where distributed delay is present in both the $v_j$ and $v_i$ terms on the right-hand
side \eqref{CS_delay2}, has not been studied before.

This paper is organized as follows. In Section \ref{sec:main} we formulate our assumptions
and the main flocking result. In Section \ref{sec:flocking} we provide its proof
divided into three steps - uniform bound on the velocities by a Lyapunov functional,
forward-backward estimates, and exponential decay of the velocity fluctuation.
Finally, in Section \ref{sec:examples} we demonstrate the applicability
of our theory to particular delay distributions - exponential,
uniform and linear. For the exponential distribution,
the flocking conditions can be resolved analytically,
leading to an explicit formula. For the other two distributions,
the satisfiability of the assumptions is tested numerically.

\section{Main result}\label{sec:main}
Let us first introduce several relevant quantities.
For $t\in\R$ we define the quadratic fluctuation of the velocities,
\(  \label{def:V}
   V(t):=\frac{1}{2}\sum_{i=1}^N \sum_{j=1}^N |v_{i}(t)-v_{j}(t)|^{2}
\)
and the quantity
\(  \label{def:D}
   D(t) := \frac12 \sum_{i=1}^N \sum_{j=1}^N \int_0^\infty {\psi}_{ij}(t-s) |{v}_{j}(t-s) - {v}_{i}(t-s)|^{2} \d\P(s).
\)
Moreover, we introduce the moments of the probability measure $\P$.
The $k$-th order moment for $k\in\N$ shall be denoted $\M_k$,
\[
   \M_k := \int_0^\infty s^k \,\d\P(s),
\]
the exponential moment (or moment generating function) $\Mexp[\kappa]$ for $\kappa\in\R$,
\[
   \Mexp[\kappa] := \int_0^\infty e^{\kappa s} \,\d\P(s).
\]
Finally, we shall need the moment $\K[\kappa]$, defined as
\(  \label{def:K}
   \K[\kappa] := \int_0^\infty s  \frac{e^{\kappa s}-1}{\kappa} \,\d\P(s).
\)
We also introduce the quantity $\L(0)$, which depends on the (constant) initial datum
and the second- and third-order moments of $\P$,
\(  \label{def:L0}
   \L(0) := V(0) + \frac{2\lambda^2 \M_3}{\sqrt{\M_2}} D(0).
\)

Our main result is the following:

\begin{theorem} \label{thm:main}
Let the communication rate $\psi=\psi(r)$ verify the assumptions \eqref{ass:psi0}--\eqref{ass:psi2}.
Let the parameter $\lambda>0$ and the probability measure $\P$ be such that
\(  \label{ass:M}
   2 \lambda \sqrt{\M_2} \leq 1.
\)
If there exists $\kappa>0$ such that the conditions
\(  \label{ass:K}
   2\lambda\sqrt{\K[\kappa]} < 1
\)
and
\(  \label{ass:Mexp}
   4\lambda\sqrt{\Mexp[\kappa]} + \alpha\sqrt{2\L(0)} < \kappa
\)
are mutually satisfied, with $\alpha>0$ given by \eqref{ass:psi2},
then the solution of the system \eqref{CS_delay1}--\eqref{CS_delay2}
subject to constant initial datum \eqref{CS_delay_IC}
exhibits flocking behavior in the sense of Definition \ref{def:flocking}.
Moreover, the quadratic velocity fluctuation $V=V(t)$ decays monotonically
and exponentially to zero as $t\to\infty$.
\end{theorem}

The above theorem deserves several comments.
First, the conditions \eqref{ass:M}--\eqref{ass:Mexp} relate the value
of the parameter $\lambda>0$, the moments of the probability measure $\P$
and the fluctuation of the initial datum through $\L(0)$.
As we shall demonstrate in Section \ref{sec:examples},
for particular choices of the distribution $\P$
they lead to systems of nonlinear inequalities in terms
of the distribution parameters and the fluctuation of the initial datum.
These can be sometimes resolved analytically,
leading to explicit flocking conditions. This is the case
for the exponential distribution, as we shall demonstrate in Section \ref{subsec:exponential}.
However, even if the nonlinear inequalities turn out to be prohibitively
complex to be treated analytically, they are well
approachable numerically. We show this for the uniform and linear distributions
in Sections \ref{subsec:uniform} and \ref{subsec:linear}.

For fixed $\lambda>0$ and $\P$, the expression \eqref{ass:Mexp}
can be interpreted as a smallness condition on the fluctuation of the initial datum
mediated through the value of $\L(0)$.
In fact, at the price of the flocking condition getting slightly more restrictive, the term $\L(0)$ in \eqref{def:L0}
can be replaced by a more intuitive expression.
Indeed, with the bound $\psi\leq 1$ given by \eqref{ass:psi0}
and the constantness of the initial datum, we have
\[
   D(0) \leq \frac12 \sum_{i=1}^N \sum_{j=1}^N \int_0^\infty |{v}_{j}(-s) - {v}_{i}(-s)|^{2} \d\P(s) = V(0).
\]
Consequently, we have
\[
   \L(0) \leq \left( 1 + \frac{2\lambda^2 \M_3}{\sqrt{\M_2}} \right) V(0)
\]
and \eqref{ass:Mexp} is satisfied whenever
\(  \label{ass:Mexpw}
   4\lambda\sqrt{\Mexp[\kappa]} + \alpha\sqrt{2\left( 1 + \frac{2\lambda^2 \M_3}{\sqrt{\M_2}} \right) V(0)} < \kappa.
\)
We find this expression more appealing since it illustrates the necessity of
smallness of the initial velocity fluctuation $V(0)$ as a condition for asymptotic flocking.

We admit that the assumption about the constantness of the initial datum on $(-\infty,0]$,
or on the interval corresponding to the support of the measure $P$, can be perceived as
too restrictive. In fact, the methods we present in this paper can be generalized to
the case of nonconstant initial data, as we demonstrated in \cite{HasMar}.
However, since this would significantly increase the technicality of our exposition,
we elect to focus on the essence of the method and thus
restrict ourselves to the constant initial datum.

We note that by the rescaling of time $t\mapsto \lambda t$, of the velocities $v_i \mapsto \lambda^{-1} v_i$
and of the probability measure $P$, the parameter $\lambda$ is eliminated from
the system \eqref{CS_delay1}--\eqref{CS_delay2}. Nonetheless,
for the purpose of compatibility with previous literature,
we shall carry out our analysis for the original form \eqref{CS_delay1}--\eqref{CS_delay2}.
The scaling invariance shall become evident in Section \ref{sec:examples},
where we shall formulate the flocking conditions in terms
of properly rescaled parameters of the probability distribution $P$
and in terms of $V(0)/\lambda^2$.

Finally, we note that the symmetry of the particle interactions $\psi_{ij} = \psi_{ji}$
implies that the total momentum is conserved along the solutions of
\eqref{CS_delay2}, i.e.,
\(  \label{mom_cons}
   \sum_{i=1}^N v_i(t) = \sum_{i=1}^N v_i(0) \qquad\mbox{for all } t\geq 0.
\)
Consequently, if the solution converges to an asymptotic velocity consensus,
then its value is determined by the mean velocity of the initial datum.

\section{Asymptotic flocking}\label{sec:flocking}

The proof of asymptotic flocking for the system \eqref{CS_delay1}--\eqref{CS_delay2}
will be carried out in three steps: First, in Section \ref{subsec:VflucB}
we shall derive an uniform bound on the quadratic velocity fluctuation
$V=V(t)$ by constructing a suitable Lyapunov functional.
Then, in Section \ref{FbEst} we prove a forward-backward estimate
on the quantity $D=D(t)$ defined in \eqref{def:D}, which states
that $D=D(t)$ changes at most exponentially locally in time.
Finally, in Section \ref{subsec:EstV} we prove the asymptotic decay
of the quadratic velocity fluctuation and boundedness of the spatial
fluctuation and so conclude the proof of Theorem \ref{thm:main}.


\subsection{Lyapunov functional and uniform bound on velocity fluctuations}\label{subsec:VflucB}

We first derive an estimate on the dissipation of the quadratic velocity fluctuation
in terms of the quantity $D=D(t)$ defined in \eqref{def:D}.

\begin{lemma} \label{lem:dV}
For any $\delta>0$ we have,
along the solutions of \eqref{CS_delay1}--\eqref{CS_delay2},
\( \label{dVest}
   \tot{}{t} V(t) \leq 2(\delta - 1)\lambda {D}(t)
      + \frac{2\lambda^3}{\delta} \int_0^\infty s \int_{t-s}^{t} {D}(\sigma) \,\d\sigma\d\P(s).
\)
\end{lemma}

\begin{proof}
We have
\[
   \tot{}{t} V(t)
   =
   \sum_{i=1}^N\sum_{j=1}^N \langle v_i-v_j, \dot v_i - \dot v_j \bigr\rangle
   =
   2N \sum_{i=1}^N \langle v_i, \dot v_i \bigr\rangle -
   2\sum_{i=1}^N\sum_{j=1}^N \langle v_i, \dot v_j \bigr\rangle
   =
   2N \sum_{i=1}^N \langle v_i, \dot v_i \bigr\rangle,
\]
where the last equality is due to the conservation of momentum \eqref{mom_cons}.
With \eqref{CS_delay2} we have
\begin{align*}
   \nonumber
   \tot{}{t} V(t) &=
   2\lambda \sum_{i=1}^N \sum_{j=1}^N \int_0^\infty \psi_{ij}(t-s) \bigl\langle  v_{i}(t),{v}_{j}(t-s)-{v}_{i}(t-s)\bigr\rangle \d\P(s) \\
   \nonumber
   &=
    2\lambda \sum_{i=1}^N \sum_{j=1}^N \int_0^\infty {\psi}_{ij}(t-s) \bigl\langle
   {v}_{i}(t-s), {v}_{j}(t-s)-{v}_{i}(t-s)\bigr\rangle \d\P(s)  \\
   &\qquad
     - 2\lambda \sum_{i=1}^N \sum_{j=1}^N \int_0^\infty {\psi}_{ij}(t-s) \bigl\langle
   {v}_{i}(t-s)-v_{i}(t), {v}_{j}(t-s)-{v}_{i}(t-s) \bigr\rangle \d\P(s).
\end{align*}
For the first term of the right-hand side we apply the standard
symmetrization trick (exchange of summation indices $i
\leftrightarrow j$, noting the symmetry of ${\psi}_{ij} ={\psi}_{ji}$),
\begin{align*}
   2\lambda \sum_{i=1}^N \sum_{j=1}^N \int_0^\infty {\psi}_{ij}(t-s) \bigl\langle
   {v}_{i}(t-s), {v}_{j}(t-s)-{v}_{i}(t-s) \bigr\rangle \d\P(s) \\
   = - \lambda \sum_{i=1}^N \sum_{j=1}^N \int_0^\infty {\psi}_{ij}(t-s) \bigl|{v}_{j}(t-s) - {v}_{i}(t-s)\bigr|^{2} \d\P(s).
\end{align*}
Therefore, we arrive at
\begin{align*}
   \tot{}{t} V(t) =
   - 2\lambda D(t)
    - 2\lambda \sum_{i=1}^N \sum_{j=1}^N \int_0^\infty {\psi}_{ij}(t-s) \bigl\langle {v}_{i}(t-s) -
   v_{i}(t), {v}_{j}(t-s) - {v}_{i}(t-s) \bigr\rangle \d P(s).
\end{align*}
For the last term we use the Young inequality with $\delta>0$ and the bound $\psi\leq 1$ by assumption \eqref{ass:psi0},
\begin{align*}
   & \left|
     2\lambda \sum_{i=1}^N \sum_{j=1}^N \int_0^\infty {\psi}_{ij}(t-s) \bigl\langle {v}_{i}(t-s) - v_{i}(t), {v}_{j}(t-s)-{v}_{i}(t-s) \bigr\rangle \d\P(s)
   \right| \\
     & \qquad \leq
   \lambda\delta \sum_{i=1}^N \sum_{j=1}^N \int_0^\infty {\psi}_{ij}(t-s) \bigl| {v}_{j}(t-s) -{v}_{i}(t-s) \bigr|^{2} \d\P(s)
   + \frac{N\lambda}{\delta} \sum_{i=1}^N \int_0^\infty  \bigl| {v}_{i}(t-s)-v_{i}(t) \bigr|^{2} \d\P(s).
\end{align*}
Hence,
\begin{align}
   \tot{}{t}V(t) \leq
     2 (\delta-1) \lambda D(t) +
       \frac{N\lambda}{\delta} \sum_{i=1}^N \int_0^\infty  \bigl| {v}_{i}(t-s)-v_{i}(t) \bigr|^{2} \d\P(s)
    \label{ke_diss}
\end{align}
Next, we use \eqref{CS_delay2} to evaluate the difference
$v_{i}(t)-{v}_{i}(t-s)$,
\begin{align} \label{vel1_del}
   v_{i}(t)-{v}_{i}(t-s) &= \int_{t-s}^{t} \tot{}{\sigma} v_{i}(\sigma)\, \d \sigma \\
    & = \frac{\lambda}{N}\sum_{j=1}^N \int_{[t-s]^+}^{t} \int_0^\infty
      {\psi}_{ij}(\sigma-\eta)({v}_{j}(\sigma-\eta)-{v}_{i}(\sigma-\eta))\, \d\P(\eta) \d \sigma,
      \nonumber
\end{align}
where $[t-s]^+ := \max\{0,t-s\}$ and we used the fact that the initial datum for
the velocity trajectories is constant.
Taking the square in \eqref{vel1_del}
and summing over $i$ we have
\begin{align} \label{vel2_del}
   \nonumber
   \sum\limits_{i=1}^N |{v}_{i}(t)-v_{i}(t-s)|^{2}
   &=\frac{\lambda^2}{N^2} \sum_{i=1}^N \left| \sum_{j=1}^N
   \int_{[t-s]^+}^{t} \int_0^\infty  {\psi}_{ij}(\sigma-\eta)({v}_{j}(\sigma-\eta)-{v}_{i}(\sigma-\eta))\, \d\P(\eta) \d\sigma
     \right|^2
   \\ \nonumber
   &\leq \frac{\lambda^2}{N}\sum_{i=1}^N\sum_{j=1}^N \left|
   \int_{[t-s]^+}^{t} \int_0^\infty  {\psi}_{ij}(\sigma-\eta)({v}_{j}(\sigma-\eta)-{v}_{i}(\sigma-\eta))\, \d\P(\eta) \d\sigma
      \right|^{2}
   \\ \nonumber
   &\leq \frac{s \lambda^2}{N}\sum_{i=1}^N\sum_{j=1}^N
     \int_{[t-s]^+}^{t} \int_0^\infty  {\psi}_{ij}(\sigma-\eta) \bigl| {v}_{j}(\sigma-\eta)-{v}_{i}(\sigma-\eta)) \bigr|^2 \, \d\P(\eta) \d\sigma
     \\
   & \leq \frac{2 s \lambda^2}{N} \int_{t-s}^{t} {D}(\sigma)\,\d\sigma.
\end{align}
The first inequality in \eqref{vel2_del} is
Cauchy-Schwartz for the sum term, i.e. $\left|\sum\limits_{i=1}^N
a_{i}\right|^2 \leq N\sum\limits_{i=1}^N |a_{i}|^2$, and the second
Cauchy-Schwartz inequality for the integral term,
together with the bound $\psi\leq 1$ imposed by assumption \eqref{ass:psi0}.
Combining \eqref{ke_diss} and \eqref{vel2_del}
directly leads to \eqref{dVest}.
\end{proof}

We now define for $t>0$ the functional
\begin{equation} \label{Lyap}
   \L(t) := V(t) + \frac{2\lambda^2}{\sqrt{\M_2}} \int_0^\infty s \int_{t-s}^{t}\int_{\theta}^{t}  {D}(\sigma) \, \d\sigma \d\theta\d\P(s),
\end{equation}
where $V=V(t)$ is the quadratic velocity fluctuation \eqref{def:V} and $D=D(t)$ is defined in \eqref{def:D}.
Note that for $t=0$ the above formula reduces to the formula \eqref{def:L0} for $\L(0)$
(recall the constantness of the initial datum).

\begin{lemma} \label{lem:Lyap}
Let the parameter $\lambda>0$ satisfy
\( \label{cond_1}
   2\lambda\sqrt{\M_2} \leq 1.
\)
Then along the solutions of \eqref{CS_delay1}--\eqref{CS_delay2}
the functional \eqref{Lyap} satisfies
\[
   \L(t)\leq \L(0)\qquad\mbox{for all } t >0.
\]
\end{lemma}

\begin{proof}
The time derivative of the second term in $\L(t)$ yields
\begin{align*} 
   \tot{}{t} \int_0^\infty s \int_{t-s}^{t}\int_{\theta}^{t}  {D}(\sigma) \, \d\sigma \d\theta\d\P(s)
      & = {D}(t) \int_0^\infty s^2 \d\P(s)  -  \int_0^\infty s \int_{t-s}^{t} {D}(\sigma) \, \d\sigma\d\P(s)
      \\
      & = \M_2 {D}(t) -  \int_0^\infty s \int_{t-s}^{t} {D}(\sigma) \, \d\sigma\d\P(s).
\end{align*}
Therefore, with the choice $\delta:=\lambda\sqrt{\M_2}$ in \eqref{dVest}
we eliminate the integral term and obtain
\(  \label{LyapDec}
   \tot{}{t} \L(t) \leq 2\lambda \left( -1 + 2\lambda\sqrt{\M_2} \right) {D}(t).
\)
We observe that the right-hand side is nonpositive if \eqref{cond_1} is verified and conclude.
\end{proof}

Consequently, if \eqref{cond_1} holds, then the velocity fluctuation $V=V(t) \leq \L(t)$ is uniformly
bounded from above by $\L(0)$ for all $t\geq 0$.

\begin{remark}
Having established the decay estimate \eqref{LyapDec}, one might attempt
to apply Barbalat's lemma to prove the desired asymptotic consensus result,
assuming merely the validity of \eqref{cond_1}.
Indeed, with the uniform bound on velocities provided by Lemma \ref{lem:Lyap}
and the properties of the interaction rate $\psi$,
one can prove that the second-order derivative $\frac{\d^2}{\d t^2} \L(t)$ is uniformly bounded in time,
which implies that $ \tot{}{t} \L(t) \to 0$ as $t\to\infty$.
This in turn gives $D(t) \to 0$ as $t\to\infty$. However, since $\psi$ is not a priori bounded from below
(and the uniform velocity bound allows for a linear in time expansion of the group in space),
this does not imply that the velocity fluctuation $V(t)$ decays asymptotically to zero.
\end{remark}

\subsection{Forward-backward estimates}\label{FbEst}

\begin{lemma} \label{lem:D_ineq}
Let the communication rate $\psi=\psi(r)$ satisfy assumption \eqref{ass:psi2}.
Then along the solutions of \eqref{CS_delay1}--\eqref{CS_delay2}, the
quantity $D(t)$ defined by \eqref{def:D} satisfies for any fixed $\eps>0$ the
inequality
\begin{equation} \label{D_ineq}
  \left| \tot{}{t} {D}(t) \right| \leq \left(2\eps + \alpha\sqrt{2\L(0)}\right) D(t) + \frac{2\lambda^2}{\eps} \int_0^\infty D(t-s) \,\d\P(s),
\end{equation}
for all $t>0$, with $\alpha>0$ given in \eqref{ass:psi2}.
\end{lemma}

\begin{proof}
For better legibility of the proof, let us adopt the notational convention that all quantities marked with a tilde
are evaluated at time point $t-s$, i.e., $\widetilde{v}_{i}:=v_i(t-s)$, $\widetilde{x}_{i}:=x_i(t-s)$,
$\widetilde{\psi}_{ij} = \psi(|\widetilde x_{i}- \widetilde x_{j}|)$, etc.
With this notation, we have
\[
   {D}(t) = \frac{1}{2} \sum_{i=1}^N \sum_{j=1}^N \int_0^\infty \widetilde\psi_{ij} |\widetilde v_{i}- \widetilde v_{j}|^{2} \, \d P(s),
\]
and differentiation in time and triangle inequality gives, for $t>0$,
\begin{align}
        \nonumber
   \left| \tot{}{t} {D}(t) \right| &\leq \frac{1}{2} \sum_{i=1}^N \sum_{j=1}^N \int_0^\infty \left| \widetilde\psi_{ij}' \left\langle
        \frac{\widetilde x_{i}- \widetilde x_{j}}{|\widetilde x_{i}-\widetilde x_{j}|},\widetilde v_{i}-\widetilde v_{j}\right\rangle \right| |\widetilde v_{i}-\widetilde v_{j}|^{2} \,\d\P(s)   \nonumber \\
      & \qquad  + \left| \sum_{i=1}^N \sum_{j=1}^N \int_0^\infty \widetilde\psi_{ij}  \left\langle \widetilde v_{i}-\widetilde v_{j},
         \tot{\widetilde v_i}{t}-\tot{\widetilde v_j}{t} \right\rangle \,\d\P(s) \right|,
  \label{dD(t)}
\end{align}
where $\psi_{ij}'=\psi'(|\widetilde x_{i}- \widetilde x_{j}|)$.
By assumption \eqref{ass:psi2}, $|\psi'(r)|\leq \alpha \psi(r)$ for $r \geq 0$,
we have for the first term of the right-hand side
\begin{align*} \nonumber
       \frac{1}{2} \sum_{i=1}^N \sum_{j=1}^N \int_0^\infty \left| \widetilde\psi_{ij}' \left\langle
        \frac{\widetilde x_{i}- \widetilde x_{j}}{|\widetilde x_{i}-\widetilde x_{j}|},\widetilde v_{i}-\widetilde v_{j}\right\rangle \right|  |\widetilde v_{i}-\widetilde v_{j}|^{2} \,\d\P(s)
            &\leq \frac{\alpha}{2} \sum_{i=1}^N \sum_{j=1}^N \int_0^\infty \widetilde \psi_{ij} |\widetilde v_{i}- \widetilde v_{j}| |\widetilde v_{i}- \widetilde v_{j}|^2 \,\d\P(s) \\
   &\leq \frac{\alpha\sqrt{2\L(0)}}{2}
   \sum_{i=1}^N \sum_{j=1}^N \int_0^\infty \widetilde \psi_{ij} |\widetilde v_{i}-\widetilde v_{j}|^2 \,\d\P(s) \nonumber \\
   & =  \alpha\sqrt{2\L(0)} D(t),
\end{align*}
where in the second inequality we used the bound
\[
   |\widetilde v_i - \widetilde v_j| \leq \sqrt{2V(t-s)} \leq \sqrt{2\L(0)},
\]
provided for $t-s>0$ by Lemma \ref{lem:Lyap}. For $t-s\leq 0$ it holds trivially due to the constantness
of the initial datum. 

For the second term of the right-hand side of \eqref{dD(t)}
we apply the symmetrization trick,
\[
    \sum_{i=1}^N \sum_{j=1}^N \int_0^\infty \widetilde\psi_{ij} \left\langle \widetilde v_{i}-\widetilde v_{j},
         \tot{\widetilde v_i}{t}-\tot{\widetilde v_j}{t} \right\rangle \,\d\P(s)
   = 2 \sum_{i=1}^N \sum_{j=1}^N \int_0^\infty \widetilde\psi_{ij} \left\langle \widetilde v_{i}-\widetilde v_{j},
         \tot{\widetilde v_i}{t} \right\rangle \,\d\P(s),
\]
and estimate using the Cauchy-Schwartz inequality with some $\eps>0$ and the bound $\psi\leq 1$ imposed by assumption \eqref{ass:psi0},
\[
   2 \left|  \sum_{i=1}^N \sum_{j=1}^N \int_0^\infty \widetilde\psi_{ij} \left\langle \widetilde v_{i}-\widetilde v_{j},
         \tot{\widetilde v_i}{t} \right\rangle \,\d\P(s) \right|
   \leq \eps \sum_{i=1}^N \sum_{j=1}^N \int_0^\infty \widetilde\psi_{ij} |\widetilde v_{i}- \widetilde v_{j}|^2 \,\d\P(s)
      +\frac{N}{\eps} \sum_{i=1}^N  \int_0^\infty \left| \tot{\widetilde v_i}{t} \right|^2 \,\d\P(s).
\]
The first term of the right-hand side is equal to $2\eps D(t)$, while for the second term, for $t-s>0$, we have
with \eqref{CS_delay2}, the Jensen inequality and the bound $\psi\leq 1$,
\begin{align*}
   \sum_{i=1}^N \left| \tot{\widetilde v_i}{t} \right|^2 &= \frac{\lambda^2}{N^2} \sum_{i=1}^N   \left| \sum_{j=1}^N
       \int_0^\infty \psi_{ij}(t-s-\sigma) (v_j(t-s-\sigma) - v_i(t-s-\sigma)) \,\d\P(\sigma) \right|^2 \\
     &\leq
     \frac{\lambda^2}{N} \sum_{i=1}^N  \sum_{j=1}^N \int_0^\infty \psi_{ij} (t-s-\sigma) \bigl| v_j(t-s-\sigma) - v_i(t-s-\sigma) \bigr|^2 \,\d\P(\sigma) \\
     & =
     \frac{2\lambda^2}{N} D(t-s).
\end{align*}
For $t-s < 0$ we have $\tot{\widetilde v_i}{t} \equiv 0$ due to the constant initial datum.

Combining the above estimates in \eqref{dD(t)}, we finally arrive at
\[
   \left| \tot{}{t} {D}(t) \right| \leq \left(2\eps + \alpha\sqrt{2\L(0)}\right) D(t) + \frac{2\lambda^2}{\eps} \int_0^\infty D(t-s) \,\d\P(s),
\]
which immediately gives \eqref{D_ineq}.
\end{proof}

The following lemma constitutes the core of the forward-backward estimate method
and was proved in \cite[Lemma 3.5]{HasMar}.
We reprint it here for the sake of the reader.

\begin{lemma} \label{lem:delayEst}
Let $y\in C(\R)$ be a nonnegative function,
continuously differentiable on $(0,\infty)$ and constant on $(-\infty,0]$.
Let the differential inequality
\(   \label{ass:y1}
    |\dot{y}(t)| \leq C_1 y(t) + C_2 \int_0^\infty y(t-s) \,\d\P(s) \qquad \text{for all } t > 0
\)
be satisfied with some constants $C_1 ,C_2 > 0$.

If there exists some $\kappa>0$ such that
\(  \label{ass:y3}
    \kappa > \max \left\{ \frac{|\dot{y}(0+)|}{y(0)}  , C_1+C_2  \Mexp[\kappa]  \right\},
\)
then the following forward-backward estimate holds for all $t>0$ and $s>0$
\(  \label{fb}
   e^{-\kappa s} y(t) < y(t-s) < e^{\kappa s} y(t).
 \)
\end{lemma}

\begin{proof}
Due to the assumed continuity of $y(t)$ and $\dot y(t)$
on $(0,\infty)$, \eqref{ass:y3} implies
that there exists $T>0$ such that
\(  \label{proof:y1}
  - \kappa < \frac{\dot{y}(t)}{y(t)} < \kappa \qquad\mbox{for all } t < T.
\)
We claim that \eqref{proof:y1} holds for all $t\in\R$, i.e., $T=\infty$.
For contradiction, assume that $T<\infty$, then again by continuity we have
\( \label{proof:y2}
   |\dot{y}(T)| = \kappa y(T).
\)
Integrating \eqref{proof:y1} on the time interval $(T-s,T)$ with $s>0$ yields
\[
   e^{-\kappa s} y(T) < y(T-s) < e^{\kappa s} y(T).
\]
Using this with \eqref{ass:y1} gives
\begin{align*}
    |\dot{y}(T)| &\leq C_1 y(T) +  C_2 \int_0^\infty y(T-s) \,\d\P(s) \\
     &< \left( C_1 +  C_2 \int_0^\infty e^{\kappa s} \,\d\P(s) \right) y(T)
     = \bigl( C_1 +  C_2 \Mexp[\kappa] \bigr) y(T).
\end{align*}
Assumption \eqref{ass:y3} gives then
\[
   |\dot{y}(T)| < \kappa y(T),
\]
which is a contradiction to \eqref{proof:y2}.
Consequently, \eqref{proof:y1} holds with $T:=\infty$,
and an integration on the interval $(t-s,t)$ implies \eqref{fb}.
\end{proof}

We now apply the result of Lemma \ref{lem:delayEst}
to derive a backward-forward estimate on the quantity
$D=D(t)$ defined in \eqref{def:D}.

\begin{lemma}\label{lem:estD3}
Let $\kappa>0$ be such that
\(   \label{ass:D2}
   \kappa > 4\lambda\sqrt{\Mexp[\kappa]} + \alpha\sqrt{2\L(0)},
\)
where $\dot{D}(0+)$ denotes the right-hand side derivative of $D(t)$ at $t=0$
along the solution of the system \eqref{CS_delay1}--\eqref{CS_delay2}.

Then for all $t>0$ and $s>0$,
\(  \label{fbV}
    e^{-\kappa s} D(t) < D(t-s) < e^{\kappa s} D(t).
 \)
\end{lemma}

\begin{proof}
We shall combine Lemma \ref{lem:D_ineq} with Lemma \ref{lem:delayEst} for $y:=D$, where we use formula \eqref{ass:y1} with
\[
   C_1 := 2\eps + \alpha\sqrt{2\L(0)}, \qquad C_2 := \frac{2\lambda^2}{\eps}.
\]
Clearly, we want to choose $\eps>0$ to minimize the expression $C_1+C_2  \Mexp[\kappa]$ in \eqref{ass:y3},
which leads to $\eps:=\lambda\sqrt{\M_\mathrm{exp}}$ and
\[
   C_1 + C_2\Mexp[\kappa] = 4\lambda\sqrt{\Mexp[\kappa]} + \alpha\sqrt{2\L(0)}.
\]
Therefore, condition \eqref{ass:y3} reads
\(   \label{ass:y3new}
   \kappa > \max \left\{ \frac{|\dot{D}(0+)|}{D(0)} , 4\lambda\sqrt{\Mexp[\kappa]} + \alpha\sqrt{2\L(0)} \right\}.
\)
To estimate the expression $\frac{|\dot{D}(0+)|}{D(0)}$, we apply Lemma \ref{lem:D_ineq} again,
this time with $t:=0$ and the optimal choice $\eps:=\lambda$.
Using the constantness of the initial datum, we have $D(s)\equiv D(0)$ for all $s<0$,
and \eqref{D_ineq} gives then
\[
   {|\dot{D}(0+)|} \leq  \left( 4\lambda + \alpha\sqrt{2\L(0)} \right) D(0).
\]
Since, by definition, $\Mexp[\kappa]\geq 1$ for $\kappa>0$, condition \eqref{ass:y3new} reduces to \eqref{ass:D2},
and we conclude.
\end{proof}

\subsection{Decay of the velocity fluctuations and flocking}\label{subsec:EstV}

In order to bound $D=D(t)$ from below by the quadratic velocity fluctuation $V=V(t)$,
we introduce the minimum interparticle interaction $\varphi = \varphi(t)$,
\(  \label{def:varphi}
   \varphi(t):=\min_{i,j=1,\cdots,N} \psi(|x_{i}(t)-x_{j}(t)|),
\)
and the position diameter
\(  \label{def:X}
   d_X(t):=\max_{i,j=1,\cdots,N} |x_i(t)-x_j(t)|.
\)
We then have the following estimate:

\begin{lemma} \label{lem:EstPhi}
Let the parameter $\lambda >0$ satisfy
\[
  \lambda\sqrt{\M_2} \leq \frac{1}{2}.
\]
Then along the solutions of \eqref{CS_delay1}--\eqref{CS_delay2}
we have
\(  \label{EstPhi}
   \varphi(t) \geq \psi\left( d_X(0) + \sqrt{2\L(0)}\,t \right) \qquad\mbox{for all } t>0.
\)
\end{lemma}

\begin{proof}
Since, by assumption, $\psi=\psi(r)$ is a nonincreasing function, we have
\(   \label{varphi_est}
   \varphi(t) = \min_{i,j=1,\cdots,N} \psi(|x_{i}(t)-x_{j}(t)|) = \psi(d_X(t)),
\)
with $d_X=d_X(t)$ defined in \eqref{def:X}.
Moreover, we have for all $i,j = 1,\cdots,N$,
\[
    \tot{}{t} |x_i - x_j|^2 \leq 2 |x_i-x_j| |v_i-v_j|,
\]
and Lemma \ref{lem:Lyap} gives
\[
   |v_i(t) - v_j(t)|^2 \leq 2V(t) \leq 2\L(0) \qquad\mbox{for all } t>0.
\]
Consequently,
\[
   \tot{}{t} |x_i - x_j|^2 \leq 2\sqrt{2\L(0)} |x_i-x_j|,
\]
and integrating in time and taking the maximum over all $i,j = 1,\cdots,N$ yields
\[
   d_X(t) \leq d_X(0) + \sqrt{2\L(0)}\, t,
\]
which combined with \eqref{varphi_est} directly implies \eqref{EstPhi}.
\end{proof}

We are now in position to provide a proof of Theorem \ref{thm:main}.
\begin{proof}
Let us recall the estimate \eqref{dVest} of Lemma \ref{lem:dV},
\[
   \tot{}{t} V(t) \leq 2(\delta - 1)\lambda {D}(t)
      + \frac{2\lambda^3}{\delta} \int_0^\infty s \int_{t-s}^{t} {D}(\sigma) \,\d\sigma\d\P(s).
\]
We apply \eqref{fbV} to the integral term
\[
   \int_{t-s}^{t} {D}(\sigma) \d\sigma = \int_{0}^{s} {D}(t-\sigma) \d\sigma < D(t) \int_{0}^{s} e^{\kappa\sigma} \d\sigma
      = D(t) \frac{e^{\kappa s}-1}{\kappa}.
\]
Consequently, we have
\begin{align*}
   \tot{}{t} V(t) \leq 2\lambda \left[ \delta - 1 + \frac{\lambda^2}{\delta} \int_0^\infty s  \frac{e^{\kappa s}-1}{\kappa}  \,\d\P(s) \right] D(t).
\end{align*}
Optimizing in $\delta>0$ gives $\delta:=\lambda\sqrt{\K[\kappa]}$, where we used
the definition \eqref{def:K} of $\K[\kappa] := \int_0^\infty s  \frac{e^{\kappa s}-1}{\kappa} \d\P(s)$, so that
\(  \label{totV}
   \tot{}{t} V(t) \leq  2\lambda \left[ 2 \lambda\sqrt{\K[\kappa]} - 1 \right] D(t).
\)
By the definition \eqref{def:varphi} of the minimal interaction $\varphi=\varphi(t)$ we have the estimate
\begin{align*}
   D(t) &= \frac12 \sum_{i=1}^N \sum_{j=1}^N \int_0^\infty {\psi}_{ij}(t-s) |{v}_{j}(t-s) - {v}_{i}(t-s)|^{2} \d\P(s) \\
      & \geq \frac12 \sum_{i=1}^N \sum_{j=1}^N \int_0^\infty {\varphi}(t-s) |{v}_{j}(t-s) - {v}_{i}(t-s)|^{2} \d\P(s) \\
      & \geq \frac{1}{2}  \psi\left( d_X(0) + \sqrt{2\L(0)}\,t \right) \sum_{i=1}^N \sum_{j=1}^N \int_0^\infty  |{v}_{j}(t-s) - {v}_{i}(t-s)|^{2} \d\P(s) \\
      &= \psi\left( d_X(0) + \sqrt{2\L(0)}\,t \right) \int_0^\infty V(t-s) \d\P(s),
\end{align*}
where for the last inequality we used \eqref{EstPhi} and the monotonicity of $\psi$,
\[
   \varphi(t-s) \geq \psi\left( d_X(0) + \sqrt{2\L(0)}\,(t-s) \right) \geq \psi\left( d_X(0) + \sqrt{2\L(0)}\,t \right).
\]
Now, if assumption \eqref{ass:K} is verified, \eqref{totV} implies that $V=V(t)$ is nonincreasing.
Thus we have $V(t-s)\geq V(t)$ for all $s>0$, and, consequently,
\(  \label{est:Dfrombelow}
  D(t) \geq \psi\left( d_X(0) + \sqrt{2\L(0)}\,t \right) V(t).
\)
Inserting into \eqref{totV} yields
\[
   \tot{}{t} V(t) \leq  2\lambda \left[ 2 \lambda\sqrt{\K[\kappa]} - 1 \right] \psi\left( d_X(0) + \sqrt{2\L(0)}\,t \right) V(t).
\]
Denoting $\omega:=-2\lambda\left[ 2 \lambda\sqrt{\K[\kappa]} - 1 \right]>0$ and integrating in time, we arrive at
\(   \label{flocking}
   V(t) \leq V(0) \exp\left(-\omega \int_0^t \psi\left(d_X(0) + \sqrt{2\L(0)} \,s \right) \d s \right).
\)
Consequently, if $\int^\infty \psi(s) \d s=\infty$,
we have the asymptotic convergence of the velocity fluctuation to zero,
$\lim_{t\to\infty} V(t) = 0$.

By assumption \ref{ass:psi1}, namely that $\psi(r) \geq Cr^{-1+\gamma}$ for all $r>R$,
we have, asymptotically for large $t>0$,
\[
    \int^t \psi\left(d_X(\tau) + \sqrt{2\L(0)} \,s \right) \d s \gtrsim t^\gamma.
\]
Consequently, from \eqref{flocking},
\[
   V(t) \lesssim \exp\left(-\omega t^\gamma \right).
\]
A slight modification of the proof of Lemma \ref{lem:EstPhi} gives
\[
   d_X(t) \leq d_X(0) + \int_0^t \sqrt{V(s)} \d s \lesssim d_X(0) + \int_0^t  \exp\left(-\omega s^\gamma/2 \right) \d s \qquad\mbox{for } t\geq 0.
\]
The integral on the right-hand side is uniformly bounded,
implying the uniform boundedness of the position diameter $d_X(t) \leq \bar d_X <+\infty$
for all $t>0$. This in turn implies $\varphi(t) \geq \psi(d_X(t)) \geq \psi(\bar d_X)$,
so that \eqref{est:Dfrombelow} is replaced by the sharper estimate
\[
     D(t) \geq \psi(\bar d_X) V(t).
\]
Thus we finally have, for all $t>0$,
\[
   \tot{}{t} V(t) \leq -\omega  \psi(\bar d_X) V(t),
\]
and conclude the exponential decay of the velocity fluctuations.
\end{proof}

\section{Examples of delay distributions} \label{sec:examples}
In this section we demonstrate how the flocking conditions \eqref{ass:M}--\eqref{ass:Mexp}
of Theorem \ref{thm:main} are resolved for particular delay distributions
- exponential, uniform on a compact interval and linear.
The conditions \eqref{ass:M}--\eqref{ass:Mexp} lead to systems of nonlinear inequalities in terms
of the distribution parameters.
For the exponential distribution they can be resolved analytically,
leading to an explicit flocking condition.
For the uniform and linear distributions they can be recast
as nonlinear minimization problems and easily resolved numerically,
using standard matlab procedures.

\subsection{Exponential distribution}\label{subsec:exponential}
We first consider the exponential distribution $\d P(s) = \mu^{-1} e^{-s/\mu} \d s$ with mean $\mu>0$.
We have for $\kappa<\mu^{-1}$,
\[
   \M_k = {k!}{\mu^k}, \qquad \Mexp[\kappa] = \frac{1}{1-\kappa\mu}, \qquad \K[\kappa] = \frac{2-\kappa\mu}{(1-\kappa\mu)^2}\mu^2.
\]
Condition \eqref{ass:M} reads $2\sqrt{2}\lambda\mu \leq 1$, and \eqref{ass:K} and \eqref{ass:Mexp} are satisfied if
there exists $\kappa>0$ such that
\[
   \frac{2\lambda\mu}{1-\kappa\mu} \sqrt{2-\kappa\mu} \leq 1,\qquad
   \qquad 4\lambda \sqrt{\frac{1}{1-\kappa\mu}} + \alpha\sqrt{2\L(0)} < \kappa.
\]
As our goal is to visualize the dependence of the flocking condition on the initial velocity fluctuation $V(0)$,
we shall work with the more restrictive version of \eqref{ass:Mexp} given by \eqref{ass:Mexpw},
which for the exponential distribution reads
\[
      4\lambda \sqrt{\frac{1}{1-\kappa\mu}}  + \alpha\sqrt{2\left(1+\frac{12(\lambda\mu)^2}{\sqrt{2}}\right) V(0)} < \kappa.
\]
Moreover, due to scaling properties, it is more convenient to investigate the flocking conditions
in terms of the product $\lambda\mu$ 
and rescale $V(0)$ by $\lambda^2$.
In this form the flocking conditions read 
\( \label{cond:exp}
   \lambda\mu \leq (2\sqrt{2})^{-1},\qquad
   \frac{2\lambda\mu}{1-\kappa\mu} \sqrt{ 2- \kappa\mu} \leq 1,\qquad
   4\lambda\mu \sqrt{\frac{1}{1-\kappa\mu}} + \alpha\lambda\mu \sqrt{2\left(1+\frac{12}{\sqrt{2}} (\lambda\mu)^2 \right) \frac{V(0)}{\lambda^2}} < \kappa\mu.
\)
The second condition in \eqref{cond:exp} is easily resolved for $\kappa\mu$,
\[  
   \kappa\mu \leq 1 - 2(\lambda\mu)^2 - 2\lambda\mu\sqrt{(\lambda\mu)^2+1},
\]
and the third condition is readily shown to be equivalent to
\[
     \sqrt[3]{32(\lambda\mu)^2} + \alpha\lambda\mu \sqrt{2\left(1+\frac{12}{\sqrt{2}} (\lambda\mu)^2\right) \frac{V(0)}{\lambda^2}} < \kappa\mu.
\]
Consequently, $\kappa\mu$ can be eliminated from \eqref{cond:exp} and we arrive at
\[
   \sqrt[3]{32(\lambda\mu)^2} + \alpha\lambda\mu \sqrt{2\left(1+\frac{12}{\sqrt{2}} (\lambda\mu)^2\right) \frac{V(0)}{\lambda^2}} <
    1 - 2(\lambda\mu)^2 - 2\lambda\mu\sqrt{(\lambda\mu)^2+1}.
\]
This is again is easily resolved for ${V(0)}/{\lambda^2}$ and we finally obtain
\(  \label{crit_exp}
   \frac{V(0)}{\lambda^2} \leq \left[ 2\left(1+\frac{12}{\sqrt{2}} (\lambda\mu)^2 \right) \right]^{-1}
      \left( \frac{1}{\alpha\lambda\mu} \left[ 1 - 2(\lambda\mu)^2 - 2\lambda\mu\sqrt{(\lambda\mu)^2+1} \right] \right)^2.
\)
The critical value of ${V(0)}/{\lambda^2}$ in dependence of $\lambda\mu$ for $\alpha:=1$ is plotted in Fig. \ref{fig:exp}.
Note that the vertical axis is in logarithmic scale.
Let us finally remark that for values of $\lambda\mu$ close to zero, the critical value of $V(0)$
behaves like $\mu^{-2}$.

\begin{figure}
\centerline{
\includegraphics[width=0.6\columnwidth]{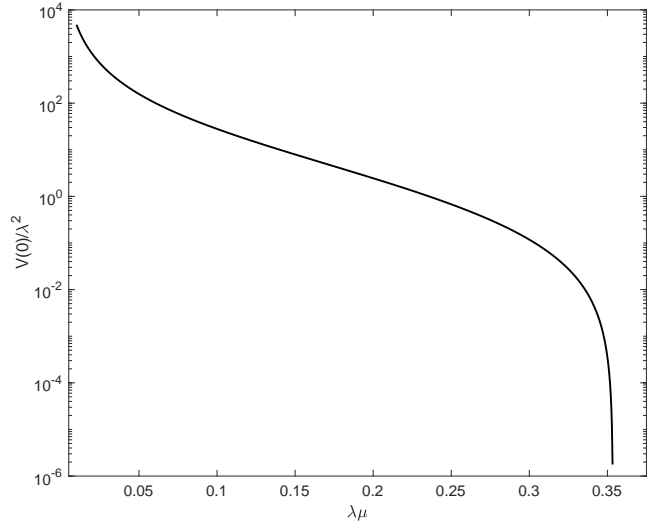}}
\caption{Critical value of ${V(0)}/{\lambda^2}$ (logarithmic scale) in dependence on $\lambda\mu\in \left[0.01, (2\sqrt{2})^{-1}\right]$ for $\alpha:=1$, as given by \eqref{crit_exp}.
}
\label{fig:exp}
\end{figure}

\subsection{Uniform distribution}\label{subsec:uniform}
Our second example is the uniform distribution on the interval $[A,B]$ with $0\leq A<B$, i.e.,
$\d P(s) = \frac{1}{B-A}\chi_{[A,B]}(s) \,\d s$. The relevant moments are
\[
   \M_k = \frac{B^{k+1}-A^{k+1}}{(k+1)(B-A)},\qquad
   \Mexp[\kappa] = \frac{ e^{\kappa B} - e^{\kappa A}}{(B-A)\kappa},\qquad
   \K[\kappa] = \frac{1}{(B-A)\kappa^2} \left( Be^{\kappa B} - Ae^{\kappa A} - \frac{e^{\kappa B} - e^{\kappa A}}{\kappa} \right).
\]
Due to the scaling relations, it is convenient to express the flocking conditions in terms of
$a:=\lambda A$, $b:=\lambda B$ and $\bar\kappa:=\kappa\lambda^{-1}$.
Condition \eqref{ass:M} reads then
\(  \label{ab}
   \frac{4}{3} \left( b^2 + ab+ a^2 \right) \leq 1,
\)
 and \eqref{ass:K}
 \(  \label{ab1}
    \frac{4}{(b-a)\bar\kappa^2} \left( be^{\bar\kappa b} - ae^{\bar\kappa a} - \frac{e^{\bar\kappa b} - e^{\bar\kappa a}}{\bar\kappa} \right) \leq 1.
 \)
Condition \eqref{ass:Mexpw} reads
 \( \label{ab2}
    4 \sqrt{\frac{ e^{\bar\kappa b} - e^{\bar\kappa a}}{(b-a)\bar\kappa}} + \alpha \sqrt{2 \left( 1+ \frac{\sqrt{3}}{2} \frac{b^3+b^2a+ba^2+a^3}{\sqrt{b^2+ab+a^2}} \right) \frac{V(0)}{\lambda^2} } < \bar\kappa.
 \)
Deciding satisfiability (in terms of $\bar\kappa>0$) of the above conditions seems to be prohibitively complex
for the analytical approach. However, the problem is well approachable numerically.
For each pair $(a,b)$ satisfying \eqref{ab}, the conditions \eqref{ab1}--\eqref{ab2} can be recast as a minimization problem
in $\bar\kappa$, and deciding satisfiability accounts to checking if the minimum is negative.
The minimization problem can be efficiently solved using the matlab procedure {\tt fminbnd}
if we provide lower and upper bounds on $\bar\kappa$.
These can be obtained analytically. Indeed, carrying our Taylor expansion of the exponentials in \eqref{ab1} we see that
\[
   \frac{4}{(b-a)\bar\kappa^2} \left( be^{\bar\kappa b} - ae^{\bar\kappa a} - \frac{e^{\bar\kappa b} - e^{\bar\kappa a}}{\bar\kappa} \right)
      \geq \frac{2(a+b)}{\bar\kappa} + \frac{4}{3} \left( a^2 + ab + b^2 \right) + \frac{\bar\kappa}{2} \left(a^3+a^2b+ab^2+b^3 \right).
\]
Combining this estimate with \eqref{ab1} gives a necessary condition for its satisfiability in terms
of explicit (in $a$ and $b$) lower and upper bounds on $\bar\kappa$, which are roots of the corresponding quadratic polynomial.
We do not print the rather lengthy algebraic expressions here; let us just mention that an immediate rough lower bound is $\bar\kappa \geq 2(a+b)$.

We carried out two numerical studies. First, we fixed the values of $\alpha:=1$ and $\frac{V(0)}{\lambda^2}:=1$
and plotted the critical value of the interval length $(b-a)$ in dependence of the value of $a>0$, see Fig. \ref{fig:uniform1}.
We see that the flocking conditions \eqref{ab1}--\eqref{ab2} are satisfiable for $a$ at most approx. $0.16$,
while for $a$ approaching zero, the interval length can go up to approx. $0.26$.
In the second study, we fixed $a:=0$ and plotted critical value of the initial fluctuation ${V(0)}/{\lambda^2}$
in dependence on the interval length $b>0$, Fig. \ref{fig:uniform2}.

\begin{figure}
\centerline{
\includegraphics[width=0.6\columnwidth]{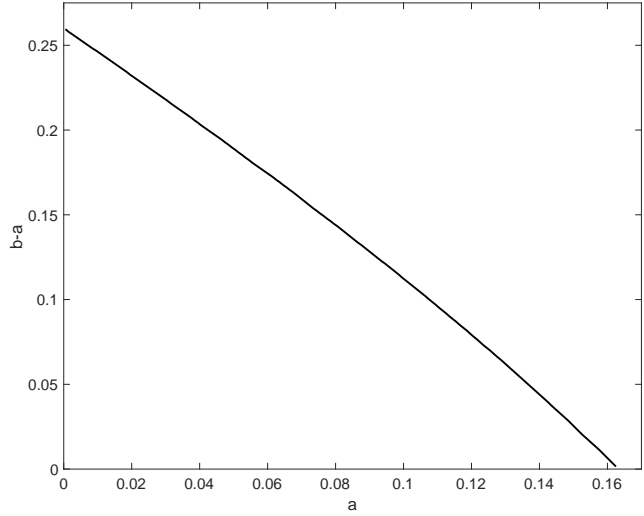}}
\caption{Critical value of the interval length $b-a$ in dependence on the value of $a>0$, obtained by numerical resolution
of the flocking condition for the uniform distribution, with $\alpha:=1$ and $\frac{V(0)}{\lambda^2}:=1$.
}
\label{fig:uniform1}
\end{figure}

\begin{figure}
\centerline{
\includegraphics[width=0.6\columnwidth]{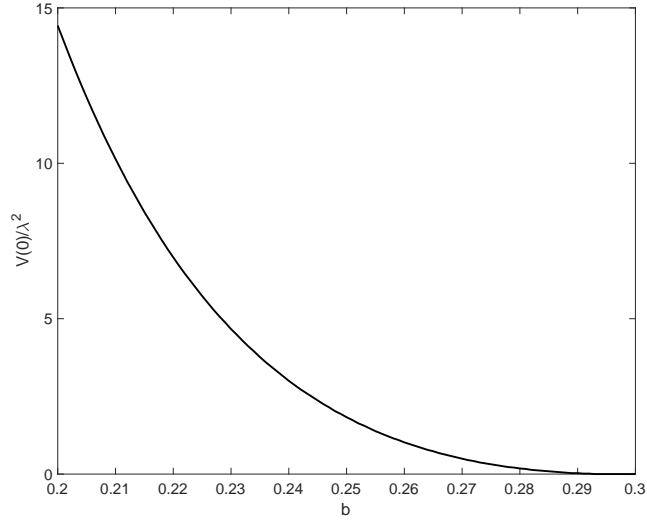}}
\caption{Critical value of the initial fluctuation ${V(0)}/{\lambda^2}$ in dependence on the value of the parameter $b\in [0.2,0.3]$,
obtained by numerical resolution of the flocking condition for the uniform distribution with $a:=0$. We set $\alpha:=1$.
}
\label{fig:uniform2}
\end{figure}

\subsection{Linear distribution}\label{subsec:linear}
Our third example is the linear distribution on the interval $[0,A]$ with $A>0$, i.e.,
$\d\P(s) = \frac{2}{A^2} [A-s]^+ \d s$, where $[A-s]^+ = \max\{0,A-s\}$.
We have
\[
   \M_k = \frac{2A^k}{(k+1)(k+2)},\qquad
   \Mexp[\kappa] = \frac{2}{\kappa A} \left( \frac{e^{A\kappa}-1}{A\kappa} -1 \right),\qquad
   \K[\kappa] = \frac{1}{\kappa} \left[ \frac{2(e^{A\kappa}+1)}{A\kappa^2} + \frac{4(1-e^{A\kappa})}{A^2\kappa^3} - \frac{A}{3} \right].
\]
Due to the scaling relations, it is again convenient to express the flocking conditions in terms of
$a:=\lambda A$, $\bar\kappa:=\kappa\lambda^{-1}$ and $V(0)/\lambda^2$.
Condition \eqref{ass:M} reads then $a\leq\sqrt{3/2}$, while \eqref{ass:K} and \eqref{ass:Mexpw} are of the form
\(  \label{cond:lin}
    \frac{4}{\bar\kappa} \left[ \frac{2(e^{a\bar\kappa}+1)}{a\bar\kappa^2} + \frac{4(1-e^{a\bar\kappa})}{a^2\bar\kappa^3} - \frac{a}{3} \right] < 1, \qquad
    4 \sqrt{ \frac{2}{\bar\kappa a} \left( \frac{e^{a\bar\kappa}-1}{a\bar\kappa} -1 \right) } + \alpha \sqrt{ 2 \left( 1 + \frac{a^2\sqrt{6}}{5}\right) \frac{V(0)}{\lambda^2}} < \bar\kappa.
\)
We are interested in the dependence of the critical value of the rescaled initial fluctuation $V(0)/\lambda^2$
on the parameter value $a$. We approach the above satisfiability problem numerically, in two steps.
First, we observe that for any fixed $a\in (0,\sqrt{3/2})$, the function
\[
   f_a(\bar\kappa) := \frac{4}{\bar\kappa} \left[ \frac{2(e^{a\bar\kappa}+1)}{a\bar\kappa^2} + \frac{4(1-e^{a\bar\kappa})}{a^2\bar\kappa^3} - \frac{a}{3} \right]
\]
is an increasing function of $\kappa>0$; this is easily seen carrying out the Taylor expansion of the exponentials.
Moreover, $\lim_{\bar\kappa\to 0+} f_a(\bar\kappa) = 2a^2/3 < 1$. Consequently, there exists $\bar{\bar\kappa}_a>0$
such that the first condition of \eqref{cond:lin} is equivalent to $\bar\kappa\in (0,\bar{\bar\kappa}_a)$.
The value of $\bar{\bar\kappa}_a$ is conveniently calculable using the matlab procedure {\tt fminsearch},
profiting from the monotonicity of the function $f_a$.
In the second step, we numerically solve the maximization problem
\[
   \max_{\bar\kappa\in (0,\bar{\bar\kappa}_a)} \left( \kappa - 4 \sqrt{ \frac{2}{\bar\kappa a} \left( \frac{e^{a\bar\kappa}-1}{a\bar\kappa} -1 \right) } \right),
\]
employing the matlab procedure {\tt fminbnd}. Then the second condition of \eqref{cond:lin} is easily resolved
for the critical value of $V(0)/\lambda^2$.
The outcome of this procedure for $\alpha:=1$ is plotted in Fig. \eqref{fig:linear}.

\begin{figure}
\centerline{
\includegraphics[width=0.6\columnwidth]{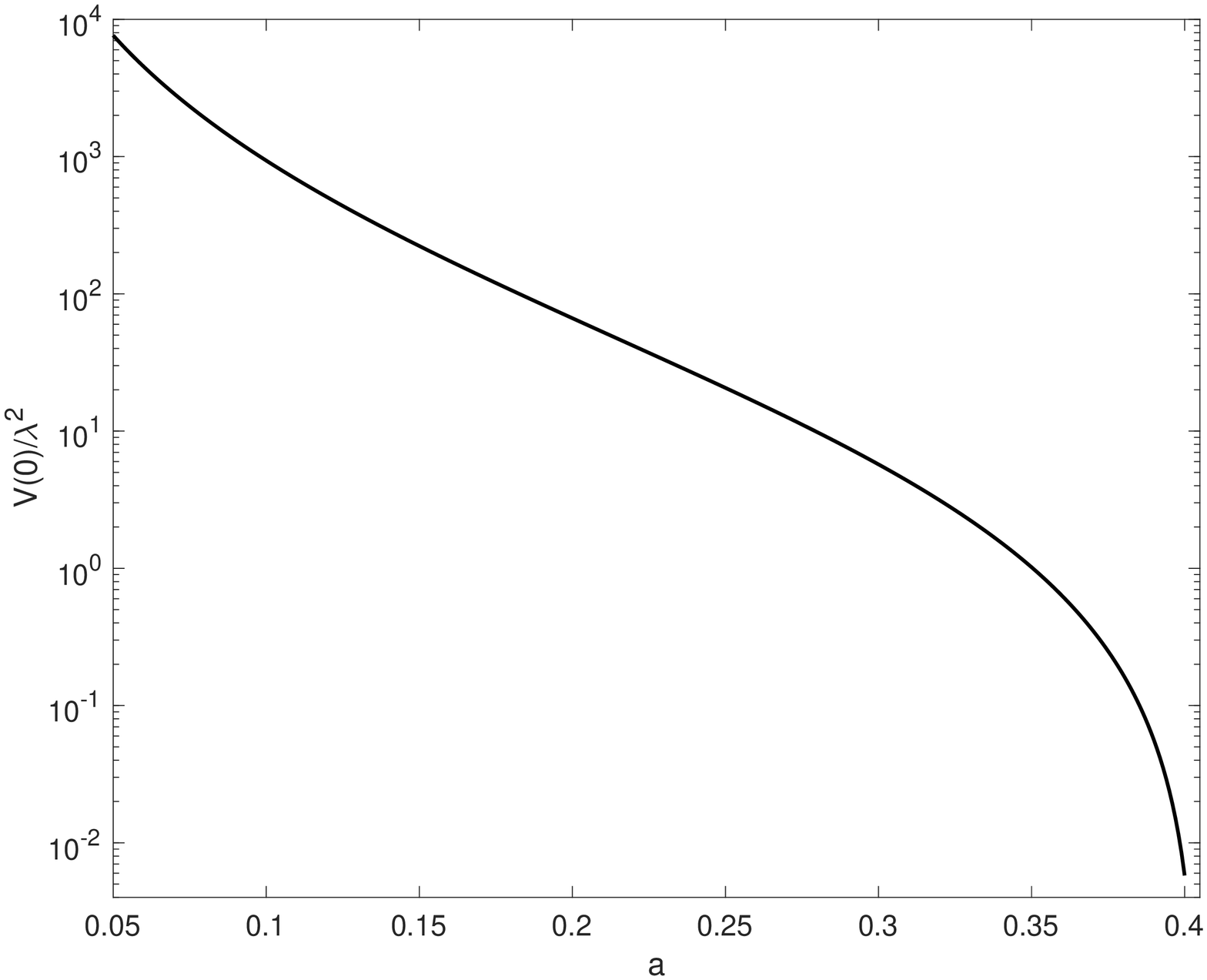}}
\caption{Critical value of the initial fluctuation ${V(0)}/{\lambda^2}$ (logarithmic scale) in dependence on the value of the parameter $a\in [0.05,0.4]$,
obtained by numerical resolution of the flocking condition for the linear distribution. We set $\alpha:=1$.
}
\label{fig:linear}
\end{figure}

\section*{Acknowledgment}
JH acknowledges the support of the KAUST baseline funds.

\quad (Jan Haskovec)
\\
\textsc{Computer, Electrical and Mathematical Sciences \&
Engineering}
\\
\textsc{King Abdullah University of Science and Technology, 23955
Thuwal, KSA}

\quad E-mail address: \textbf{jan.haskovec@kaust.edu.sa}
\\ \\

\quad (Ioannis Markou)
\\
\textsc{Institute of Applied and Computational Mathematics
(IACM-FORTH)}
\\
\textsc{N. Plastira 100, Vassilika Vouton GR - 700 13, Heraklion,
Crete, Greece}

\quad E-mail address: \textbf{ioamarkou@iacm.forth.gr}

\end{document}